\documentclass{amsart}
\usepackage{amsfonts}

\theoremstyle{plain}
\newtheorem{theorem}{Theorem}
\newtheorem{lemma}{Lemma}
\newtheorem{corollary}{Corollary}
\newtheorem{proposition}{Proposition}

\theoremstyle{definition}

\theoremstyle{remark}
\newtheorem{remark}{Remark}

\numberwithin{equation}{section}
\input{tcilatex}

\begin{document}
\title[Poisson-Mehler summation formula]{Around Poisson--Mehler summation
formula}
\author{Pawe\l\ J. Szab\l owski}
\address{Department of Mathematics and Information Sciences,\\
Warsaw University of Technology\\
ul. Koszykowa 75, 00-662 Warsaw, Poland}
\email{pawel.szablowski@gmail.com}
\date{February 27, 2012}
\subjclass{Primary 33D45, 41A10; Secondary 62H0, 42C05}
\keywords{$q-$Hermite, big $q-$Hermite, Al-Salam--Chihara, orthogonal
polynomials, Poisson-Mehler summation formula. Orthogonal polynomials on the
plane.}

\begin{abstract}
We study polynomials in $x$ and $y$ of degree $n+m:\allowbreak \left\{
Q_{m,n}\left( x,y|t,q\right) \right\} _{n,m\geq 0}$ that appeared recently
in the following identity: $\gamma _{m,n}\left( x,y|t,q\right) \allowbreak
=\allowbreak \gamma _{0,0}\left( x,y|t,q\right) \allowbreak Q_{m,n}\left(
x,y|t,q\right) $ where $\gamma _{m,n}\left( x,y|t,q\right) \allowbreak
=\allowbreak \sum_{i\geq 0}\frac{t^{i}}{\left[ i\right] _{q}!}H_{i+n}\left(
x|q\right) H_{m+i}(y|q)$, $\allowbreak $\newline
$\left\{ H_{n}\left( x|q\right) \right\} _{n\geq -1}$ are the so-called $q-$%
Hermite polynomials (qH). In particular we show that the spaces $span\left\{
Q_{i,n-i}\left( x,y|t,q\right) :i=0,\ldots ,n\right\} _{n\geq 0}$ are
orthogonal with respect to a certain measure (two-dimensional $(t,q)-$Normal
distribution) on the square $\left\{ (x,y):|x|,|y|\leq 2/\sqrt{1-q}\right\}
. $ We study structure of these polynomials expressing them with the help of
the so-called Al-Salam--Chihara (ASC) polynomials and showing that they are
rational functions of parameters $t$ and $q$. We use them in various
infinite expansions that can be viewed as simple generalization of the
Poisson-Mehler summation formula. Further we use them in the expansion of
the reciprocal of the right hand side of the Poisson-Mehler formula.
\end{abstract}

\maketitle

\section{Introduction and auxiliary results}

\subsection{Preface}

We consider various generalizations of the celebrated Poisson--Mehler
formula (see e.g. \cite{IA}, (13.1.24) or \cite{Andrews1999}, (10.11.17)): 
\begin{equation}
\sum_{n\geq 0}\frac{\rho ^{n}}{\left[ n\right] _{q}!}H_{n}\left( x|q\right)
H_{n}\left( y|q\right) \allowbreak =\allowbreak \frac{\left( \rho
^{2}\right) _{\infty }}{\prod_{j=0}^{\infty }\omega \left( x\sqrt{1-q}/2,y%
\sqrt{1-q}/2|\rho q^{j}\right) },  \label{PM}
\end{equation}%
where $\left\{ H_{n}\right\} _{n\geq 0}$ denote $q-$Hermite polynomials and $%
\omega \left( x,y|t\right) $ are certain polynomials symmetric in $x$ and $y$
of degree two. These polynomials as well as symbols $\left[ n\right] _{q}!$
and $\left( \rho ^{2}\right) _{\infty }$are defined and explained in
Sections \ref{notacja} and \ref{wiel}. There exist many proofs of (\ref{PM})
(e.g. see \cite{IA}, \cite{Andrews1999}, \cite{bressoud}, \cite{Szabl-Exp}).
Recently in \cite{Szabl-AW} certain generalization of (\ref{PM}) has been
proved by the author. It was used in calculating moments of the so called
Askey--Wilson distribution.

In the paper we consider functions

\begin{equation}
\gamma _{i,j}\left( x,y|\rho ,q\right) =\sum_{n\geq 0}\frac{\rho ^{n}}{\left[
n\right] _{q}!}H_{n+i}\left( x|q\right) H_{n+j}\left( y|q\right) .
\label{gamma}
\end{equation}%
for all $i,j\geq 0.$ It was shown by the author in \cite{Szabl-Exp} (Lemma
3) that:%
\begin{equation}
\gamma _{i,j}\left( x,y|\rho ,q\right) \allowbreak =\allowbreak
Q_{i,j}\left( x,y|\rho ,q\right) \gamma _{0,0}\left( x,y|\rho ,q\right) ,
\label{uPM}
\end{equation}%
where $Q_{i,j}\left( x,y|\rho ,q\right) $ is a certain polynomial in $x,$ $y$
of degree $i+j.$ Hence (\ref{uPM}) can be viewed as a generalization of (\ref%
{PM}).

The main object of the paper is to study the properties and later the r\^{o}%
le of the polynomials $Q_{i,j}\left( x,y|\rho ,q\right) $ in obtaining
various expansions that can be viewed as either generalizations of (\ref{PM}%
) or expansions more or less directly related to this formula.

In particular we find generating function of these polynomials, we express
them as linear combinations of polynomials belonging to families of
polynomials of one variable.

We also analyze the measure (the so-called $(\rho ,q)-2Normal$ measure) on
the square $S\left( q\right) \times S\left( q\right) $ with the density
defined by (\ref{f2D}) below, that can be easily constructed from the
densities of measures that make $q-$Hermite and the so-called
Al-Salam--Chihara polynomials orthogonal and which can viewed as a
generalization of bivariate Normal distribution. Interval $S\left( q\right) $
is defined by (\ref{podst1}). The probabilistic aspects of this distribution
were presented in \cite{Szab5}. We point out the r\^{o}le of the polynomials 
$Q_{n,m}$ in further analysis of this measure. In particular we introduce
spaces of functions of two variables 
\begin{equation}
\Lambda _{n}\allowbreak (x,y|\rho ,q)\allowbreak =\allowbreak span\left\{
Q_{i,n-i}(x,y,|\rho ,q),\allowbreak i=0,\ldots ,n\right\} ,n\geq 0
\label{podprz}
\end{equation}%
and show that they are orthogonal with respect to $(\rho ,q)-2Normal$
measure. Hence these spaces form the direct sum decomposition of the space
of functions that are square integrable with respect to $(\rho ,q)-2Normal$
measure.

Further we use these polynomials to obtain various infinite expansions. In
particular we obtain an expansion of the reciprocal of the right hand side
of (\ref{PM}) in an infinite series. In \cite{Szabl-Exp}, (formula 5.3) one
such expansion was presented. The expansions was non-symmetric in $x$ and $y$
(for each finite sum). This time the expansion is symmetric in $x$ and $y.$

Among other possible views one can look at the results of paper as the
generalization of the results of the two papers of Van der Jeugt et al. \cite%
{Jeugt97}, \cite{Jeugt98}. The authors of these papers introduced
convolutions of known families of classical orthogonal polynomials such as
Hermite or Laguerre considered at two variables thus obtaining bivariate
polynomials. They applied their results in Lie algebra and its
generalizations.

Our "convolutions" concern generalizations of Hermite polynomials ($q$%
-Hermite, and Al-Salam--Chihara). As possible applications we mean the ones
in analysis, two dimensional orthogonal polynomials theory or probability.

Since in our paper appear kernels built of mostly $q-$Hermite and
Al-Salam--Chihara one should remark that some of the technics used in the
proofs resemble those used in e.g. \cite{Ismail97}. But by no means results
are the same.

The paper is organized as follows. In the next two Subsections (i.e. \ref%
{notacja} and \ref{wiel}) we provide simple introduction to $q-$series
theory presenting typical notation used and presenting a few typical
families of the so called basic orthogonal polynomials. The word basic comes
from the base which is the parameter in most cases denoted by $q$. We do
this since notation and terminology used in $q-$series theory is somewhat
specific and not widely known to those not working within this field. We are
also purposely not using notation based on hypergeometric series since it is
mostly known to specialists of special functions theory. We believe that the
results presented in the paper can be applied in various fields of
traditional analysis like the theory of Fourier expansions, theory of
reproducing kernels, orthogonal polynomials theory and last but not least
probability theory. Then in Section \ref{glow} we present our main results,
open questions and remarks are in Section \ref{open} while less interesting
laborious proofs are in Section \ref{dowody}.

\subsection{Notation\label{notacja}}

We use notation traditionally used in the so called $q-$series theory. Since
not all readers are familiar with it we will recall now this notation.

$q$ is a parameter. We will assume that $-1<q\leq 1$ unless otherwise
stated. Let us define $\left[ 0\right] _{q}\allowbreak =\allowbreak 0;$ $%
\left[ n\right] _{q}\allowbreak =\allowbreak 1+q+\ldots +q^{n-1}\allowbreak ,
$ $\left[ n\right] _{q}!\allowbreak =\allowbreak \prod_{j=1}^{n}\left[ j%
\right] _{q},$ with $\left[ 0\right] _{q}!\allowbreak =1$ and%
\begin{equation*}
\QATOPD[ ] {n}{k}_{q}\allowbreak =\allowbreak \left\{ 
\begin{array}{ccc}
\frac{\left[ n\right] _{q}!}{\left[ n-k\right] _{q}!\left[ k\right] _{q}!} & 
, & n\geq k\geq 0 \\ 
0 & , & otherwise%
\end{array}%
\right. .
\end{equation*}%
It will be useful to use the so called $q-$Pochhammer symbol for $n\geq 1:$%
\begin{eqnarray*}
\left( a;q\right) _{n} &=&\prod_{j=0}^{n-1}\left( 1-aq^{j}\right) , \\
\left( a_{1},a_{2},\ldots ,a_{k};q\right) _{n}\allowbreak  &=&\allowbreak
\prod_{j=1}^{k}\left( a_{j};q\right) _{n},
\end{eqnarray*}%
with $\left( a;q\right) _{0}=1$. Often $\left( a;q\right) _{n}$ as well as $%
\left( a_{1},a_{2},\ldots ,a_{k};q\right) _{n}$ will be abbreviated to $%
\left( a\right) _{n}$ and $\left( a_{1},a_{2},\ldots ,a_{k}\right) _{n},$ if
it will not cause misunderstanding.

It is easy to notice that for $\left\vert q\right\vert <1$ we have $\left(
q\right) _{n}=\left( 1-q\right) ^{n}\left[ n\right] _{q}!$ and \newline
$\QATOPD[ ] {n}{k}_{q}\allowbreak =$\allowbreak $\allowbreak \left\{ 
\begin{array}{ccc}
\frac{\left( q\right) _{n}}{\left( q\right) _{n-k}\left( q\right) _{k}} & ,
& n\geq k\geq 0 \\ 
0 & , & otherwise%
\end{array}%
\right. $. \newline
Notice that $\left[ n\right] _{1}\allowbreak =\allowbreak n,\left[ n\right]
_{1}!\allowbreak =\allowbreak n!,$ $\QATOPD[ ] {n}{k}_{1}\allowbreak
=\allowbreak \binom{n}{k},$ $\left( a;1\right) _{n}\allowbreak =\allowbreak
\left( 1-a\right) ^{n}$ and $\left[ n\right] _{0}\allowbreak =\allowbreak
\left\{ 
\begin{array}{ccc}
1 & if & n\geq 1 \\ 
0 & if & n=0%
\end{array}%
\right. ,$ $\left[ n\right] _{0}!\allowbreak =\allowbreak 1,$ $\QATOPD[ ] {n%
}{k}_{0}\allowbreak =\allowbreak 1,$ $\left( a;0\right) _{n}\allowbreak
=\allowbreak \left\{ 
\begin{array}{ccc}
1 & if & n=0 \\ 
1-a & if & n\geq 1%
\end{array}%
\right. .$

In the sequel we shall also use the following useful notation:%
\begin{eqnarray}
S\left( q\right) &=&\left\{ 
\begin{array}{ccc}
\lbrack -\frac{2}{\sqrt{1-q}},\frac{2}{\sqrt{1-q}}] & if & \left\vert
q\right\vert <1 \\ 
\mathbb{R} & if & q=1%
\end{array}%
\right. ,  \label{podst1} \\
I_{A}\left( x\right) &=&\left\{ 
\begin{array}{ccc}
1 & if & x\in A \\ 
0 & if & x\notin A%
\end{array}%
\right. .  \label{podst2}
\end{eqnarray}

\subsection{Polynomials\label{wiel}}

\subsubsection{$q-$Hermite}

Let $\left\{ H_{n}\left( x|q\right) \right\} _{n\geq 0}$ denote the family
of the so called $q-$ Hermite (briefly qH) polynomials. That is the one
parameter family of orthogonal polynomials satisfying the following three
term recurrence:%
\begin{equation}
H_{n+1}\left( x|q\right) \allowbreak =\allowbreak xH_{n}\left( x|q\right) - 
\left[ n\right] _{q}H_{n-1}\left( x|q\right) ,  \label{_H}
\end{equation}%
with $H_{-1}\left( x|q\right) \allowbreak =\allowbreak 0$ and $H_{0}\left(
x|q\right) =\allowbreak 1.$ In fact in the literature (see e.g. \cite%
{Andrews1999}, \cite{IA}, \cite{KLS}) function more often the re-scaled
versions of these polynomials. Namely more often appear under the name of $%
q- $Hermite polynomials the following polynomials $\left\{ h_{n}\left(
x|q\right) \right\} _{n\geq 0}$ defined by their three term recurrence: 
\begin{equation}
h_{n+1}\left( x|q\right) =2xh_{n}\left( x|q\right) -(1-q^{n})h_{n-1}\left(
x|q\right) ,  \label{_h}
\end{equation}%
with $h_{-1}\left( x|q\right) \allowbreak =\allowbreak 0$ and $h_{0}\left(
x|q\right) \allowbreak \allowbreak =\allowbreak 1$. These polynomials are
related to one another by the relationship $\forall n\geq -1$: 
\begin{equation}
H_{n}\left( x|q\right) =\frac{h_{n}\left( x\sqrt{1-q}/2|q\right) }{\left(
1-q\right) ^{n/2}},  \label{scal}
\end{equation}%
for $\left\vert q\right\vert <1$. For $q=1$ $h_{n}(x|1)=2^{n}x^{n}$ while $%
H_{n}\left( x|1\right) =H_{n}\left( x\right) $, where polynomials $%
H_{n}\left( x\right) $ are the so called probabilistic Hermite polynomials
i.e. monic\footnote{%
i.e. polynomials with leading coefficient equal to $1.$} polynomials
orthogonal with respect to $\exp \left( -x^{2}/2\right) .$ Observe further
that $h_{n}\left( x|0\right) \allowbreak =\allowbreak U_{n}\left( x\right) $
where $U_{n}$ denotes the so called Chebyshev polynomial of the second kind
(for details see e.g. \cite{Andrews1999}).

The polynomials $H_{n}$ have a nice probabilistic interpretation (see e.g. 
\cite{Szabl-AW}) and besides constitute really the generalization of the
ordinary Hermite polynomials. That is why we will use them in this paper.
The results presented here can be easily adopted and expressed in terms of
polynomials $h_{n}.$

The generating function of these polynomials is given by the following
formula that is in fact adapted to our setting formula (14.26.1) of \cite%
{KLS} 
\begin{equation}
\varphi _{H}\left( x|\rho ,q\right) =\sum_{n\geq 0}\frac{\rho ^{n}}{\left[ n%
\right] _{q}!}H_{n}\left( x|q\right) =\frac{1}{\prod_{j=0}^{\infty }v\left( x%
\sqrt{1-q}/2|\rho q^{j}\sqrt{1-q}\right) },  \label{fi_H}
\end{equation}%
convergent for $\left\vert \rho (1-q)\right\vert <1,$ $x\in S\left( q\right)
,$ where we denoted 
\begin{equation}
v\left( x|t\right) \allowbreak =\allowbreak 1-2xt+t^{2}.  \label{_v}
\end{equation}

Let us observe that $\forall x\in \lbrack -1,1],$ $t\in \mathbb{R}:v\left(
x|t\right) \geq 0.$

Besides adapting formula (14.26.2) of \cite{KLS} to our setting we have:%
\begin{equation}
\int_{S\left( q\right) }H_{n}\left( x|q\right) H_{m}\left( x|q\right)
f_{N}\left( x|q\right) dx=\left[ n\right] _{q}!\delta _{mn},  \label{ort}
\end{equation}%
with 
\begin{equation}
f_{N}\left( x|q\right) =\frac{\sqrt{(1-q)(4-(1-q)x^{2})}\left( q\right)
_{\infty }}{2\pi }\prod_{j=1}^{\infty }l\left( x\sqrt{1-q}/2|q^{j}\right) ,
\label{fN}
\end{equation}%
for $x\in S\left( q\right) ,$ where 
\begin{equation}
l\left( x|a\right) \allowbreak =\allowbreak (1+a)^{2}-4ax^{2}.  \label{_ll}
\end{equation}%
One shows that 
\begin{eqnarray}
\lim_{q\rightarrow 1^{-}}f_{N}\left( x|q\right)  &=&\frac{1}{\sqrt{2\pi }}%
\exp \left( -x^{2}/2\right) ,  \label{zb1} \\
\lim_{q\rightarrow 1^{-}}\frac{1}{\prod_{j=0}^{\infty }v\left( x|\rho
q^{j}\right) } &=&\exp \left( x\rho -x^{2}/2\right) .  \label{zb2}
\end{eqnarray}

Apart from $q-$Hermite polynomials we will need the so called big $q-$%
Hermite (briefly bqH) polynomials $\left\{ H_{n}\left( x|a,q\right) \right\}
_{n\geq -1}$ with $a\in \mathbb{R}$. They are defined through their three
term recurrence: 
\begin{equation}
H_{n+1}\left( x|a,q\right) \allowbreak =\allowbreak (x-aq^{n})H_{n}\left(
x|a,q\right) -[n]_{q}H_{n-1}\left( x|a,q\right) ,  \label{bH}
\end{equation}%
with $H_{-1}\left( x|a,q\right) \allowbreak =\allowbreak 0,$ $H_{0}\left(
x|a,q\right) \allowbreak =\allowbreak 1.$ To support intuition let us remark
that $H_{n}(x|a,1)\allowbreak =\allowbreak H_{n}\left( x-a\right) $ and $%
H_{n}\left( x|a,0\right) \allowbreak =\allowbreak U_{n}\left( x/2\right)
-aU_{n-1}\left( x/2\right) .$

One knows its relationship with the $q-$Hermite polynomials:%
\begin{equation*}
H_{n}\left( x|a,q\right) \allowbreak =\sum_{k=0}^{n}\QATOPD[ ] {n}{k}%
_{q}(-a)^{k}q^{\binom{k}{2}}H_{n-k}\left( x|q\right) ,
\end{equation*}%
and that (see e.g. \cite{KLS}, (14.18.2) with an obvious modification for
polynomials $H_{n}$): 
\begin{eqnarray*}
\int_{S\left( q\right) }H_{n}\left( x|a,q\right) H_{m}\left( x|a,q\right)
f_{bN}\left( x|a,q\right) dx &=&\left[ n\right] _{q}!\delta _{mn}, \\
\sum_{n\geq 0}\frac{t^{n}}{\left[ n\right] _{q}!}H_{n}\left( x|a,q\right) 
&=&\varphi _{H}\left( x|t,q\right) \left( (1-q)at\right) _{\infty },
\end{eqnarray*}%
where 
\begin{equation}
f_{bN}\left( x|a,q\right) \allowbreak =\allowbreak f_{N}\left( x|q\right)
\varphi _{H}\left( x|a,q\right) .  \label{bigH}
\end{equation}

We will need the following Lemma concerning another relationship between
polynomials $H_{n}\left( x|q\right) $ and $H_{n}\left( x|a,q\right) .$

\begin{lemma}
\label{basic} Let us define for $\forall n\geq 0;\allowbreak x\in S\left(
q\right) ;\allowbreak (1-q)t^{2}<1:\eta _{n}\left( x|t,q\right) \allowbreak
=\allowbreak \sum_{j\geq 0}\frac{t^{j}}{\left[ j\right] _{q}!}H_{j+n}\left(
x|q\right) .$ Then 
\begin{equation*}
\eta _{n}\left( x|t,q\right) \allowbreak =\allowbreak H_{n}\left(
x|t,q\right) \varphi _{H}(x|t,q),
\end{equation*}%
where $H_{n}\left( x|t,q\right) $ is the bqH polynomial defined by (\ref{bH}%
).
\end{lemma}

\begin{proof}
In a version with continuous $q-$Hermite polynomials $h$ defined by (\ref{_h}%
) and $h_{n}(x|t,q)$ are the big $q-$Hermite polynomials as defined in \cite%
{KLS} (14.18.4) this formula has been proved as a particular case in \cite%
{Szabl-peculiar} (2.1). We notice that $\eta _{0}\left( x|t,q\right)
\allowbreak =\allowbreak \varphi _{H}(x|t,q).$ To switch to polynomials $%
H_{n}$ using (\ref{scal}) is elementary.
\end{proof}

\begin{remark}
Let us remark that Carlitz in \cite{Carlitz72} considered similar shifted
characteristic functions of the form $\sum_{j\geq 0}\frac{t^{j}}{(q)_{j}}%
w_{n+j}(x|q)\ $with Rogers--Szeg\"{o} polynomials $w_{n}$ (see discussion
below following formula (\ref{expansion})). From this result of Carlitz one
can also deduce assertion of Lemma \ref{basic}.
\end{remark}

\subsubsection{Al-Salam--Chihara}

Next family of polynomials that we are going to consider depends on $2$
(apart from $q)$ parameters denoted by $a$ and $b$, that satisfy the
following three term recurrence (see e.g. \cite{KLS},(14.8.4)):

\begin{equation}
A_{n+1}\left( x|a,b,q\right) =(2x-(a+b)q^{n})A_{n}\left( x|a,b,q\right)
-(1-abq^{n-1})(1-q^{n})A_{n-1}\left( x|y,\rho ,q\right) ,  \label{_asc}
\end{equation}%
with $A_{-1}\left( x|a,b,q\right) \allowbreak =\allowbreak 0,$ $A_{0}\left(
x|a,b,q\right) \allowbreak =\allowbreak 1.$ These polynomials will be called
Al-Salam--Chihara polynomials $\left\{ A_{n}\left( x|a,b,q\right) \right\}
_{n\geq -1}$ (briefly ASC). It follows from Favard's theorem that the
measure that makes these polynomials orthogonal is positive if $\forall
n\geq 1\allowbreak :\allowbreak (1-abq^{n-1})(1-q^{n})\allowbreak \geq
\allowbreak 0,$ which for $\left\vert q\right\vert \leq 1$ is reduced to $%
\left\vert ab\right\vert \leq 1.$

In the sequel in fact we will consider these polynomials with complex
parameters forming a conjugate pair and also re-scaled. Namely we will take $%
a\allowbreak =\allowbreak \frac{\sqrt{1-q}}{2}\rho (y\allowbreak
-\allowbreak i\sqrt{\frac{4}{1-q}-y^{2}}),$ $b\allowbreak =\allowbreak \frac{%
\sqrt{1-q}}{2}\rho (y\allowbreak +\allowbreak i\sqrt{\frac{4}{1-q}-y^{2}}),$
with $y\in S\left( q\right) $ and $\left\vert \rho \right\vert <1.$ More
precisely we will consider polynomials $\left\{ P_{n}\left( x|y,\rho
,q\right) \right\} _{n\geq 0}$ defined by:%
\begin{equation*}
A_{n}\left( x\frac{\sqrt{1-q}}{2}|a,b,q\right) /(1-q)^{n/2}\allowbreak
=\allowbreak P_{n}\left( x|y,\rho ,q\right) .
\end{equation*}%
One can easily notice that $a\allowbreak +\allowbreak b\allowbreak
=\allowbreak \rho y\sqrt{1-q},$ $ab\allowbreak =\allowbreak \rho ^{2}$ and
thus that the polynomials $P_{n}$ satisfy the following three term
recurrence:%
\begin{equation}
P_{n+1}\left( x|y,\rho ,q\right) =(x-\rho yq^{n})P_{n}\left( x|y,\rho
,q\right) -[n]_{q}(1-\rho ^{2}q^{n-1})P_{n-1}\left( x|y,\rho ,q\right) ,
\label{Pn}
\end{equation}%
with $P_{-1}\left( x|y,\rho ,q\right) \allowbreak =\allowbreak 0,$ $%
P_{0}\left( x|y,\rho ,q\right) \allowbreak =\allowbreak 1.$

\begin{remark}
To support intuition let us remark (following e.g. \cite{Szabl-AW}) that $%
P_{n}\left( x|y,\rho ,1\right) =H_{n}\left( \frac{x-\rho y}{\sqrt{1-\rho ^{2}%
}}\right) \left( 1-\rho ^{2}\right) ^{n/2}.$ On the other hand $P_{n}\left(
x|y,\rho ,0\right) \allowbreak =\allowbreak U_{n}\left( x/2\right)
\allowbreak -\allowbreak \rho yU_{n-1}\left( x/2\right) \allowbreak
+\allowbreak \rho ^{2}U_{n-2}\left( x/2\right) ,$ where $U_{n}\left(
x\right) $ denotes Chebyshev polynomial of the second kind.
\end{remark}

It is known see e.g. \cite{KLS}, (formula (14.8.13) adapted to our setting), 
\cite{bms} or \cite{Szabl-AW} that the polynomials $P_{n}$ have the
following generating function:

\begin{equation*}
\varphi _{P}\left( x|y,\rho ,t,q\right) =\sum_{n\geq 0}\frac{t^{n}}{\left[ n%
\right] _{q}!}P_{n}\left( x|y,\rho ,q\right) =\prod_{j=0}^{\infty }\frac{%
v\left( y\sqrt{1-q}/2|\rho tq^{j}\sqrt{1-q}\right) }{v\left( x\sqrt{1-q}%
/2|tq^{j}\sqrt{1-q}\right) },
\end{equation*}%
convergent for $\left\vert t\sqrt{1-q}\right\vert ,\left\vert \rho
\right\vert <1,$ $x,y\in S\left( q\right) $.

We also have (see e.g. \cite{Szabl-AW}) or :%
\begin{equation}
\int_{-1}^{1}P_{n}(x|y,\rho ,q)P_{m}\left( x|y,\rho ,q\right) f_{CN}\left(
x|y,\rho ,q\right) \allowbreak =\allowbreak \delta _{nm}\left[ n\right]
_{q}!(\rho ^{2})_{n},  \label{pkw}
\end{equation}%
where%
\begin{equation*}
f_{CN}\left( x|y,\rho ,q\right) \allowbreak =\allowbreak f_{N}\left(
x|q\right) \frac{\left( \rho ^{2}\right) _{\infty }}{\prod_{j=0}^{\infty
}\omega \left( x\sqrt{1-q}/2,y\sqrt{1-q}/2|\rho q^{j}\right) },
\end{equation*}%
with 
\begin{equation}
\omega \left( x,y|\rho \right) =\left( 1-\rho ^{2}\right) ^{2}-4\rho (1+\rho
^{2})xy+4\rho ^{2}\left( x^{2}+y^{2}\right) .  \label{_W}
\end{equation}

\begin{remark}
It was shown in \cite{SzablKer}(Lemma 1, (v)) that for $\left\vert
q\right\vert <1$ function $\left\vert f_{CN}(x|y,\rho ,q)/f_{N}\left(
x|q\right) \right\vert $ is bounded both from below and above hence square
integrable on the square $S(q)\times S\left( q\right) $ with respect to the
measure $f_{N}\left( x|q\right) f_{CN}\left( x|y,\rho ,q\right) dxdy.$ This
on its side will guarantee existence and convergence of some Fourier
expansions considered in the next section.
\end{remark}

We will call the densities $f_{N}$ and $f_{CN}$ respectively $q-$Normal and $%
(q,\rho )-$Conditional Normal. The names are justified by the nice
probabilistic interpretations of these densities presented e.g. in \cite{Bo}%
, \cite{Bryc2001M}, \cite{Bryc2001S}, \cite{bms}, \cite{Szabl-AW} or \cite%
{Szab5}. Besides apart from (\ref{zb1}) we also have:%
\begin{equation*}
\lim_{q\rightarrow 1^{-}}f_{CN}\left( x|y,\rho ,q\right) =\exp \left( -\frac{%
(x-\rho y)^{2}}{2(1-\rho ^{2})}\right) /\sqrt{2\pi (1-\rho ^{2})}.
\end{equation*}

We end up this section by recalling an auxiliary simple result that will be
used in following sections many times. It has been formulated and proved in 
\cite{Szabl-peculiar} Proposition 2.

\begin{proposition}
\label{pomoc}Let $\sigma _{n}\allowbreak \left( \rho |q\right) =\allowbreak
\sum_{j\geq 0}\frac{\rho ^{j}}{\left[ j\right] _{q}!}\xi _{n+j}$ for $%
\left\vert \rho \right\vert <1,$ $-1\allowbreak <q\allowbreak \leq 1$ and
certain sequence $\left\{ \xi _{m}\right\} _{m\geq 0}$ such that $\sigma _{n}
$ exists for every $n.$ Then 
\begin{equation}
\sigma _{n}\left( \rho q^{m}|q\right) \allowbreak =\allowbreak
\sum_{k=0}^{m}\left( -1\right) ^{k}\QATOPD[ ] {m}{k}_{q}q^{\binom{k}{2}%
}\left( 1-q\right) ^{k}\rho ^{k}\sigma _{n+k}(\rho |q).  \label{sigmanm}
\end{equation}
\end{proposition}

\begin{remark}
Notice that this Proposition is trivially true also for both $q\allowbreak
=\allowbreak 0$ and $q\allowbreak =\allowbreak 1.$
\end{remark}

\section{Main Results\label{glow}}

One of our main interests in this paper are the generalizations of the
Poisson-Mehler formula (\ref{PM}).

It is well known that convergence in (\ref{PM}) takes place for $x,y\in
S\left( q\right) ,$ $\left\vert \rho \right\vert <1$ and for $\left\vert
q\right\vert <1$ is uniform. For $q\allowbreak =\allowbreak 1$ we have
almost uniform convergence.

As a immediate corollary of Proposition \ref{pomoc} we have:

\begin{corollary}
For $\left\vert q\right\vert <1$ we have:%
\begin{eqnarray}
\gamma _{i,j}\left( x,y|\rho q^{m},q\right) \allowbreak  &=&\allowbreak
\sum_{k=0}^{m}\left( -1\right) ^{k}\QATOPD[ ] {m}{k}_{q}q^{\binom{k}{2}%
}\left( 1-q\right) ^{k}\rho ^{k}\gamma _{i+k,j+k}\left( x,y|\rho ,q\right) ,
\label{_l} \\
H_{i}\left( x|q\right) H_{j}\left( y|q\right) \allowbreak  &=&\allowbreak
\sum_{k\geq 0}\left( -1\right) ^{k}q^{\binom{k}{2}}\frac{\rho ^{k}}{\left(
q\right) _{k}}\gamma _{i+k,j+k}\left( x,y|\rho ,q\right) ,  \label{_lnsk}
\end{eqnarray}%
where $\gamma _{i,j}(x,y|\rho ,q)$ is defined by (\ref{gamma}). (\ref{_l})
is also true trivially for $q\allowbreak =\allowbreak 1.$
\end{corollary}

\begin{proof}
First assertion we get by applying directly (\ref{sigmanm}) by setting $%
\sigma _{i,j}\allowbreak =\allowbreak \gamma _{i,j}.$ Second assertion we
get by passing in the first one with $m$ to infinity and then noticing
firstly that $\lim_{m\rightarrow \infty }\QATOPD[ ] {m}{k}_{q}\allowbreak
=\allowbreak \frac{1}{\left[ k\right] _{q}!}$ and finally that $\frac{\left(
1-q\right) ^{k}}{\left[ k\right] _{q}!}\allowbreak =\allowbreak \frac{1}{%
\left( q\right) _{k}}.$
\end{proof}

Now let us turn to polynomials $Q_{i,j}\left( x,y|\rho ,q\right) $ defined
by (\ref{uPM}). It was shown in \cite{Szabl-AW} that for all $-1<q\leq 1,$ $%
\left\vert \rho \right\vert <1,$ $x,y=\allowbreak x,y\in \mathbb{R}$ : 
\begin{equation}
Q_{i,j}\left( x,y|\rho ,q\right) \allowbreak =\allowbreak
\sum_{s=0}^{j}(-1)^{s}q^{\binom{s}{2}}\QATOPD[ ] {j}{s}_{q}\rho
^{s}H_{j-s}\left( y|q\right) P_{i+s}\left( x|y,\rho ,q\right) /\left( \rho
^{2}\right) _{i+s},  \label{expansion}
\end{equation}%
and $Q_{i,j}\left( x,y|\rho ,q\right) \allowbreak =\allowbreak Q_{j,i}\left(
y,x|\rho ,q\right) .$

\begin{remark}
It has to be remarked that Carlitz in \cite{Carlitz72} considered the sum $%
\xi _{k,j}(x,y|\rho ,q)\allowbreak =\allowbreak \sum_{n\geq 0}\frac{\rho ^{n}%
}{(q)_{n}}w_{n+k}\left( x|q\right) w_{n+j}\left( y|q\right) ,$ where $%
w_{n}(x|q)$ are the so called Rogers--Szeg\"{o} polynomials related to
polynomials $h_{n}(x|q)$ by the formula: $h_{n}(x|q)\allowbreak =\allowbreak
e^{ni\theta }w_{n}(e^{-2i\theta }|q)$ with $x\allowbreak =\allowbreak \cos
\theta ,$ $i\allowbreak -\allowbreak $imaginary unit. Indeed it turned out
that functions $\xi _{k,j}$ also have the property that%
\begin{equation*}
\xi _{k,j}(x,y|\rho ,q)\allowbreak =\allowbreak \nu _{k,j}(x,y|\rho ,q)\xi
_{0,0}(x,y|\rho ,q),
\end{equation*}%
where $\nu _{k,j}$ are polynomials of degree $k+j$ in $x$ and $y.$ However
to show that $\nu _{k,j}(e^{-i\theta },e^{-i\eta }|\rho ,q)$ can be
expressed as $Q_{k,j}(\cos \theta ,\cos \eta |\rho ,q)$ is not an easy task.
Discussion on this subject is in \cite{Szab-bAW}. In particular see the
proof of Proposition 5.
\end{remark}

In particular we have 
\begin{equation}
Q_{k,0}\left( x,y|\rho ,q\right) \allowbreak =\allowbreak P_{k}\left(
x|y,\rho ,q\right) /\left( \rho ^{2}\right) _{k}.  \label{qi0}
\end{equation}

To analyze further properties of polynomials $Q_{k,j}$ let us introduce the
following $2$ dimensional density defined for $S^{2}\left( q\right)
\allowbreak \overset{df}{=}S\left( q\right) \times S\left( q\right) .$ 
\begin{equation}
f_{2D}\left( x,y|\rho ,q\right) \allowbreak =\allowbreak f_{CN}\left(
x|y,\rho ,q\right) f_{N}\left( y|q\right) .  \label{f2D}
\end{equation}

Measure that has density $f_{2D}$ will be called $(\rho ,q)-$bivariate
Normal (briefly $(\rho ,q)-2N$). Obviously $f_{2D}\left( x,y|\rho ,q\right)
\allowbreak =\gamma _{0,0}\left( x,y|\rho ,q\right) f_{N}\left( x|q\right)
f_{N}\left( y|q\right) .$ Its applications in theories of probability and
Markov stochastic processes have been presented in \cite{Szab5} and \cite%
{Szab-OU-W}.

Here below we give another interpretation of the polynomials $Q_{n,m}$ in
particular its connection with the big $q-$Hermite polynomials.

\begin{proposition}
\label{f_ch}For $\left\vert q\right\vert <1,$ $\left\vert \rho \right\vert
<1,$ $x,y\in \mathbb{R}$ we have:

i) $\forall i,j,m,k,i+j\neq m+k:$ 
\begin{gather*}
\int_{S^{2}\left( q\right) }Q_{i,j}\left( x,y|\rho ,q\right) Q_{m,k}\left(
x,y|\rho ,q\right) f_{2D}\left( x,y|\rho ,q\right) dxdy=0, \\
\int_{S^{2}\left( q\right) }Q_{n-j,j}\left( x,y|q\right) Q_{n-k,k}\left(
x,y|\rho ,q\right) f_{2D}\left( x,y|\rho ,q\right) dxdy\allowbreak = \\
\left( -1\right) ^{k-j}\frac{\rho ^{k-j}q^{\binom{k-j}{2}}\left[ j\right]
_{q}!\left[ n-j\right] _{q}!}{\left( \rho ^{2}\right) _{n}}%
\sum_{s=0}^{j}q^{s(s-1)+ns}\QATOPD[ ] {k}{k-j+s}_{q}\QATOPD[ ] {n-j+s}{s}%
_{q}\rho ^{2s}\left( \rho ^{2}q^{n-j+s}\right) _{j-s}.
\end{gather*}

ii) 
\begin{equation*}
\sum_{n,m\geq 0}\frac{t^{n}s^{m}}{\left[ n\right] _{q}!\left[ m\right] _{q}!}%
Q_{n,m}\left( x,y|\rho ,q\right) \allowbreak =\allowbreak \frac{f_{bH}\left(
x|t,q\right) f_{bH}\left( y|s,q\right) }{f_{2D}(x,y|\rho ,q)}\sum_{k\geq 0}%
\frac{\rho ^{k}}{\left[ k\right] _{q}!}H_{k}\left( x|t,q\right) H_{k}\left(
y|s,q\right) ,
\end{equation*}%
where function $f_{bH}$ is defined by (\ref{bigH}). The above mentioned
formulae are also true for $q=1.$

iii) $\allowbreak \forall m\geq 0:$ 
\begin{gather}
Q_{i,j}\left( x,y|\rho q^{m},q\right) \prod_{i=0}^{m-1}\omega \left( x\sqrt{%
1-q}/2,y\sqrt{1-q}/2|\rho q^{i}\right) \allowbreak =\allowbreak 
\label{QnaQ} \\
\left( \rho ^{2}\right) _{2m}\sum_{k=0}^{m}\left( -1\right) ^{k}\QATOPD[ ] {m%
}{k}_{q}q^{\binom{k}{2}}\left( 1-q\right) ^{k}\rho ^{k}Q_{i+k,j+k}\left(
x,y|\rho ,q\right) ,
\end{gather}%
where polynomial $\omega $ is defined by (\ref{_W}). In particular we have:%
\begin{equation}
\prod_{j=0}^{n-1}\omega \left( x\sqrt{1-q}/2,y\sqrt{1-q}/2|\rho q^{j}\right)
=\left( \rho ^{2}\right) _{2n}\sum_{k=0}^{n}\left( -1\right) ^{k}\QATOPD[ ] {%
n}{k}_{q}q^{\binom{k}{2}}\left( 1-q\right) ^{k}\rho ^{k}Q_{k,k}\left(
x,y|\rho ,q\right) ,  \label{il}
\end{equation}%
and 
\begin{equation}
q^{\binom{n}{2}}\rho ^{n}(1-q)^{n}Q_{n,n}\left( x,y|\rho ,q\right)
\allowbreak =\allowbreak \sum_{k=0}^{n}\left( -1\right) ^{k}q^{\binom{n-k}{2}%
}\QATOPD[ ] {n}{k}_{q}\frac{\prod_{j=0}^{k-1}\omega \left( x\sqrt{1-q}/2,y%
\sqrt{1-q}/2|\rho q^{j}\right) }{\left( \rho ^{2}\right) _{2k}},  \label{qkk}
\end{equation}%
with understanding that $\prod_{j=0}^{k-1}$ for $k\allowbreak =\allowbreak 0$
is equal to $1$.
\end{proposition}

\begin{proof}
Is shifted to section \ref{dowody}.
\end{proof}

Our main results follow in fact directly the results presented above.

\begin{theorem}
\label{main}Either for $\left\vert q\right\vert <1;x,y\in S\left( q\right)
;\left\vert \rho \right\vert <1$ we have:

i) 
\begin{equation*}
H_{i}\left( x|q\right) H_{j}\left( y|q\right) \frac{\prod_{k=0}^{\infty
}\omega \left( x\sqrt{1-q}/2,y\sqrt{1-q}/2|\rho q^{k}\right) }{\left( \rho
^{2}\right) _{\infty }}\allowbreak =\allowbreak \sum_{k=0}^{\infty }\left(
-1\right) ^{k}q^{\binom{k}{2}}\frac{\rho ^{k}}{\left[ k\right] _{q}!}%
Q_{i+k,j+k}\left( x,y|\rho ,q\right) .
\end{equation*}%
In particular we get:

ii)%
\begin{eqnarray}
1/\sum_{n\geq 0}\frac{\rho ^{n}}{\left[ n\right] _{q}!}H_{n}\left(
x|q\right) H_{n}\left( y|q\right)  &=&\frac{\prod_{k=0}^{\infty }\omega
\left( x\sqrt{1-q}/2,y\sqrt{1-q}/2|\rho q^{k}\right) }{\left( \rho
^{2}\right) _{\infty }}\allowbreak   \notag \\
&=&\allowbreak \sum_{k=0}^{\infty }\left( -1\right) ^{k}q^{\binom{k}{2}}%
\frac{\rho ^{k}}{\left[ k\right] _{q}!}Q_{k,k}\left( x,y|\rho ,q\right) .
\label{recip}
\end{eqnarray}

The last formula is valid also for $\allowbreak x,y\in \mathbb{R},$ $%
q\allowbreak =\allowbreak 1$ and $\left\vert \rho \right\vert <1/2.$
\end{theorem}

\begin{proof}
To get i) we pass in (\ref{QnaQ}) with $m$ to infinity noting by (\ref%
{expansion}) and (\ref{Pn}) that $Q_{n,m}(x,y|0,q)\allowbreak =\allowbreak
H_{n}\left( x|q\right) H_{m}\left( y|q\right) $. On the way we observe that $%
\lim_{m\rightarrow \infty }\QATOPD[ ] {m}{k}_{q}\allowbreak =\allowbreak 
\frac{1}{\left( q\right) _{k}}\allowbreak =\allowbreak (1-q)^{-k}\frac{1}{%
\left[ k\right] _{q}!}.$ As far as the case $q\allowbreak =\allowbreak 1$ is
concerned denote by $g_{N}(x,y,\rho )$ density of the bivariate Normal
density with parameters $\sigma _{1}=\sigma _{2}\allowbreak =\allowbreak 1,$
correlation coefficient $\rho .$ Then notice that function $\exp (-\frac{1}{2%
}(x^{2}+y^{2}))/g_{N}(x,y,\rho )$ is square integrable on the plane with
respect to $g_{N}(x,y,\rho )$ if $\left\vert \rho \right\vert <1.$
\end{proof}

\section{Open problems and comments\label{open}}

\begin{remark}
The non-symmetric kernels constructed of bqH polynomials were given in \cite%
{suslov96}. Formula ii) of Proposition \ref{f_ch} gives its new
interpretation. Besides, recall that these kernels were expressed using
basic hypergeometric function $_{3}\phi _{2}.$ Expansion on the left hand
side of Proposition \ref{f_ch}ii) gives new outlook on the properties of
this function.

Notice also that for $q\allowbreak =\allowbreak 1$ we have $\eta
(x|t,1)=\exp (xt-\frac{t^{2}}{2}),$ $H_{n}\left( x|t,1\right) \allowbreak
=\allowbreak H_{n}\left( x-t\right) $ and 
\begin{equation*}
\sum_{n\geq 0}\frac{\rho ^{n}}{n!}H_{n}\left( x\right) H_{n}(y)=\exp (\frac{%
x^{2}}{2}-\frac{(x-\rho y)^{2}}{2(1-\rho ^{2})}),
\end{equation*}%
hence characteristic function of polynomials $Q_{i,j}$ can be calculated
explicitly.

Similarly for $q=0$ we have $\eta (x|t,0)\allowbreak =\allowbreak \frac{1}{%
1-xt+t^{2}}$ (characteristic function of the Chebyshev polynomials) and $%
H_{n}\left( x|t,0\right) \allowbreak =\allowbreak U_{n}\left( x/2\right)
-tU_{n-1}\left( x/2\right) $ (see e.g. \cite{Szabl-rev}) hence also in this
case we can get explicit form of the characteristic function of polynomials $%
Q_{i,j}.$
\end{remark}

\begin{remark}
First of all notice that the left hand side of (\ref{recip}) is equal to $%
1/\gamma _{0,0}\left( x,y|\rho ,q\right) \allowbreak =\allowbreak
f_{N}\left( x|q\right) /f_{CN}\left( x|y,\rho ,q\right) $ and that it is a
symmetric ( with respect to $x$ and $y)$ function. In \cite{Szabl-Exp} there
was presented (formula 5.3) an expansion of this function involving
polynomials $P_{n}$ and certain polynomials related to $q-$Hermite ones. The
expansion was non-symmetric for every partial sum. Thus we get another
expansion of known important special function.
\end{remark}

\begin{remark}
Assertion i) of Proposition \ref{f_ch} states that polynomials $Q_{n,m}$ and 
$Q_{i,j}$ are orthogonal with respect to two dimensional measure $\mu _{2D}$
with the density given by (\ref{f2D}) if only the $n+m\allowbreak \neq
\allowbreak i+j$. Let us define space $\mathcal{L\allowbreak =\allowbreak }%
L_{2}\left( S^{2}\left( q\right) ,\mathcal{B},\mu _{2D}\right) $ of
functions $f:S^{2}\left( q\right) \longrightarrow \mathbb{R}$ square
integrable with respect to the measure $\mu _{2D}.$ Do polynomials $Q_{m,n}$
constitute a base of this space? It seems that yes. We can define subspaces
of $\allowbreak \Lambda _{m}\allowbreak =\allowbreak \allowbreak span\left\{
Q_{m,0},\ldots ,Q_{0,m}\right\} $ of polynomials that are linear
combinations of polynomials $Q_{i,j}$ such that $i+j\allowbreak =\allowbreak
m.$ Subspaces $\Lambda _{m}$ are mutually orthogonal. Besides following
argument that polynomials are dense in $\mathcal{L}$ we deduce that $%
\mathcal{L=}\dbigoplus\limits_{n=0}^{\infty }\Lambda _{n}.$ What is the
orthogonal base of $\mathcal{L}?$ We have calculated covariances between
polynomials $Q_{i,j}$ from $\Lambda _{m}$ following (\ref{expansion}) and (%
\ref{pkw}). Thus we can follow Gram-Schmidt orthogonalization procedure
within the spaces $\Lambda _{m}$. Is the union of orthogonal bases of $%
\Lambda _{m}$ an orthogonal base of $\mathcal{L}$? Again it seems that yes.
It would be interesting to find this base.
\end{remark}

\begin{remark}
In 2001 W\"{u}nsche in \cite{Wunsche} considered Hermite and Laguerre
polynomials on the plane. He has not however related his Hermite polynomials
to any particular measure on the plane. In particular he defined Hermite
polynomials depending on parameters forming a $2x2$ matrix. This matrix is
however not connected in any way to the covariance matrix of the measure
with respect to which these polynomials are supposed to be orthogonal.

On the other hand definition of polynomials $Q_{i,j}$ depends heavily on the
measure with the density $f_{2D}$. For $q\allowbreak =\allowbreak 1$
following (\ref{expansion}), we have 
\begin{equation*}
Q_{i,j}\left( x,y|\rho ,1\right) \allowbreak =\allowbreak
\sum_{k=0}^{j}\left( -1\right) ^{k}\binom{j}{k}H_{j-k}\left( y\right)
H_{k+i}\left( \frac{x-\rho y}{\sqrt{1-\rho ^{2}}}\right) /\left( \sqrt{%
1-\rho ^{2}}\right) ^{k+i}.
\end{equation*}%
Hence polynomials $Q_{i,j}\left( x,y,\rho ,1\right) $ are in fact another
(different from that of W\"{u}nsche's) family of two dimensional
generalization of Hermite polynomials.
\end{remark}

\section{Proofs\label{dowody}}

\begin{proof}[Proof of Proposition \protect\ref{f_ch}]
i) We use (\ref{expansion}), assume that $i>m.$ We have:%
\begin{gather*}
\int_{S^{2}\left( q\right) }Q_{i,j}\left( x,y|q\right) Q_{m,k}\left(
x,y|\rho ,q\right) f_{2D}\left( x,y|\rho ,q\right) dxdy\allowbreak = \\
\sum_{s=0}^{j}\sum_{t=0}^{k}\left( -1\right) ^{s+t}q^{\binom{s}{2}}q^{\binom{%
t}{2}}\QATOPD[ ] {j}{s}_{q}\QATOPD[ ] {k}{t}_{q}\rho ^{s+t}\frac{1}{(\rho
^{2})_{i+s}\left( \rho ^{2}\right) _{m+t}}\int_{S\left( q\right)
}H_{j-s}\left( y\right) H_{k-t}\left( y|q\right) f_{N}\left( y|q\right)
\allowbreak  \\
\times \allowbreak \int_{S\left( q\right) }P_{i+s}\left( x|y,\rho ,q\right)
P_{m+t}\left( x|y,\rho ,q\right) f_{CN}\left( x|y,\rho ,q\right)
dx\allowbreak dy= \\
\left( -1\right) ^{i-m}\rho ^{i-m}\sum_{s=0\vee m-i}^{j\wedge k+m-i}q^{%
\binom{s}{2}+\binom{i-m+s}{2}}\QATOPD[ ] {j}{s}_{q}\QATOPD[ ] {k}{i+s-m}%
_{q}\times  \\
\rho ^{2s}\frac{\left[ i+s\right] _{q}!}{\left( \rho ^{2}\right) _{i+s}}%
\int_{S\left( q\right) }H_{j-s}\left( y\right) H_{k+m-i-s}\left( y|q\right)
f_{N}\left( y|q\right) dy\allowbreak =\allowbreak 0,
\end{gather*}%
\newline
$\allowbreak \allowbreak $ if $j-s\neq k+m-i-s$ i.e. if $j+i\neq k+m.$

Assuming that $k\geq j$ we get%
\begin{gather*}
\int_{S^{2}\left( q\right) }Q_{n-j,j}\left( x,y|q\right) Q_{n-k,k}\left(
x,y|\rho ,q\right) f_{2D}\left( x,y|\rho ,q\right) dxdy\allowbreak = \\
\sum_{s=0}^{j}\sum_{t=0}^{k}\left( -1\right) ^{s+t}q^{\binom{s}{2}}q^{\binom{%
t}{2}}\QATOPD[ ] {j}{s}_{q}\QATOPD[ ] {k}{t}_{q}\rho ^{s+t}\frac{1}{(\rho
^{2})_{n-j+s}\left( \rho ^{2}\right) _{n-k+t}}\int_{S\left( q\right)
}H_{j-s}\left( y\right) H_{k-t}\left( y|q\right) f_{N}\left( y|q\right)
\allowbreak \\
\times \allowbreak \int_{S\left( q\right) }P_{n-j+s}\left( x|y,\rho
,q\right) P_{n-k+t}\left( x|y,\rho ,q\right) f_{CN}\left( x|y,\rho ,q\right)
dx\allowbreak dy= \\
\left( -1\right) ^{k-j}\rho ^{k-j}\sum_{s=0}^{j}q^{\binom{s}{2}+\binom{k-j+s%
}{2}}\QATOPD[ ] {j}{s}_{q}\QATOPD[ ] {k}{k-j+s}_{q}\rho ^{2s}\frac{\left[
n-j+s\right] _{q}!}{\left( \rho ^{2}\right) _{n-j+s}}\int_{S\left( q\right)
}H_{j-s}\left( y\right) H_{j-s}\left( y|q\right) f_{N}\left( y|q\right)
dy\allowbreak \\
=\left( -1\right) ^{k-j}\rho ^{k-j}\sum_{s=0}^{j}q^{\binom{s}{2}+\binom{k-j+s%
}{2}}\QATOPD[ ] {j}{s}_{q}\QATOPD[ ] {k}{k-j+s}_{q}\rho ^{2s}\frac{\left[
n-j+s\right] _{q}!}{\left( \rho ^{2}\right) _{n-j+s}}\left[ j-s\right] _{q}!
\\
=\left( -1\right) ^{k-j}\frac{\rho ^{k-j}q^{\binom{k-j}{2}}\left[ j\right]
_{q}!\left[ n-j\right] _{q}!}{\left( \rho ^{2}\right) _{n}}%
\sum_{s=0}^{j}q^{s(s-1)+ns}\QATOPD[ ] {k}{k-j+s}_{q}\QATOPD[ ] {n-j+s}{s}%
_{q}\rho ^{2s}\left( \rho ^{2}q^{n-j+s}\right) _{j-s}
\end{gather*}%
we use here $s(s-1)/2+(s+n)(s-1+n)/2-s(s-1)-n(n-1)/2=\allowbreak ns$

ii) We have 
\begin{eqnarray*}
\sum_{i\geq 0,j\geq 0}\frac{s^{i}t^{j}}{\left[ i\right] _{q}!\left[ j\right]
_{q}!}Q_{i,j}\left( x,y|\rho ,q\right) \allowbreak &=&\allowbreak \frac{1}{%
\gamma _{0,0}\left( x,y|\rho ,q\right) }\sum_{i\geq 0,j\geq 0}\frac{%
t^{i}s^{j}}{\left[ i\right] _{q}!\left[ j\right] _{q}!}\sum_{n\geq 0}\frac{%
\rho ^{n}}{\left[ n\right] _{q}!}H_{i+n}\left( x|q\right) H_{n+j}\left(
y|q\right) \\
&=&\allowbreak \sum_{n\geq 0}\frac{\rho ^{n}}{\left[ n\right] _{q}!}%
\sum_{j=0}^{\infty }\frac{s^{j}}{\left[ j\right] _{q}!}H_{n+j}\left(
y|q\right) \sum_{i\geq 0}\frac{t^{i}}{\left[ i\right] _{q}!}H_{n+i}\left(
x|q\right) .
\end{eqnarray*}%
Now we use Lemma $\allowbreak $\ref{basic} twice and get%
\begin{gather*}
\sum_{i\geq 0,j\geq 0}\frac{s^{i}t^{j}}{\left[ i\right] _{q}!\left[ j\right]
_{q}!}Q_{i,j}\left( x,y|\rho ,q\right) \allowbreak =\allowbreak \frac{%
\varphi _{H}(x|t,q)\varphi _{H}\left( y|s,q\right) }{\gamma _{0,0}(x,y|\rho
,q)}\sum_{n\geq 0}\frac{\rho ^{n}}{\left[ n\right] _{q}!}H_{n}\left(
x|t,q\right) H_{n}(y|s,q) \\
=\frac{1}{(\rho ^{2})_{\infty }}\prod_{j=0}^{\infty }\frac{\omega \left( x%
\sqrt{1-q}/2,y\sqrt{1-q}/2|\rho q^{j}\right) }{v\left( x\sqrt{1-q}/2|t\sqrt{%
1-q}q^{j}\right) v\left( y\sqrt{1-q}/2|s\sqrt{1-q}q^{j}\right) }\sum_{n\geq
0}\frac{\rho ^{n}}{\left[ n\right] _{q}!}H_{n}\left( x|t,q\right)
H_{n}(y|s,q).
\end{gather*}

iii) First we notice that from (\ref{uPM}) it follows that for $x,y\in
S\left( q\right) ;\rho ^{2}<1,-1<q\leq 1:$ 
\begin{eqnarray*}
\gamma _{i,j}\left( x,y|\rho q^{m},q\right) \allowbreak  &=&\allowbreak
Q_{i,j}\left( x,y|\rho q^{m},q\right) \frac{\left( \rho ^{2}q^{2m}\right)
_{\infty }}{\prod_{i=0}^{\infty }\omega \left( x\sqrt{1-q}/2,y\sqrt{1-q}%
/2|\rho q^{m+i}\right) }\allowbreak  \\
&=&\allowbreak Q_{i,j}\left( x,y|\rho q^{m},q\right) \frac{%
\prod_{i=0}^{m-1}\omega \left( x\sqrt{1-q}/2,y\sqrt{1-q}/2|\rho q^{i}\right) 
}{\left( \rho ^{2}\right) _{2m}}\gamma _{0,0}\left( x,y,\rho ,q\right) ,
\end{eqnarray*}%
and also that $\gamma _{0,0}\left( x,y|\rho ,q\right) \allowbreak
=\allowbreak \frac{\left( \rho ^{2}\right) _{\infty }}{\prod_{i=0}^{\infty
}\omega \left( x\sqrt{1-q}/2,y\sqrt{1-q}/2|\rho q^{i}\right) }.$ Then we
apply (\ref{_l}) to $\gamma _{i,j}$ above and then use (\ref{sigmanm}) and
cancel out $\gamma _{0,0}$ on both sides of (\ref{_l}). $\allowbreak $%
Finally we observe that on both sides we have polynomials hence one can
extend the identity for all values of the variables. To get other formula of
this assertion we argue by induction checking that the equality is true for $%
n\allowbreak =\allowbreak 0.$ Then we put (\ref{il}) into (\ref{qkk}) and
get:%
\begin{gather*}
\sum_{k=0}^{n}\left( -1\right) ^{k}q^{\binom{n-k}{2}}\QATOPD[ ] {n}{k}_{q}%
\frac{\prod_{i=0}^{k-1}\omega \left( x\sqrt{1-q}/2,y\sqrt{1-q}/2|\rho
q^{i}\right) }{\left( \rho ^{2}\right) _{2k}}\allowbreak =\allowbreak  \\
\sum_{k=0}^{n}\left( -1\right) ^{k}q^{\binom{n-k}{2}}\QATOPD[ ] {n}{k}%
_{q}\sum_{j=0}^{k}(-1)^{j}\QATOPD[ ] {k}{j}_{q}q^{\binom{j}{2}}\rho
^{j}Q_{j,j}\left( x,y|\rho ,q\right) \allowbreak = \\
\sum_{j=0}^{n}(-1)^{j}q^{\binom{j}{2}}\QATOPD[ ] {n}{j}_{q}\rho
^{j}Q_{j,j}\left( x,y|\rho ,q\right) \sum_{k=j}^{n}\left( -1\right) ^{k}q^{%
\binom{n-k}{2}}\QATOPD[ ] {n-j}{k-j}_{q}= \\
\sum_{j=0}^{n}(-1)^{j}q^{\binom{j}{2}}\QATOPD[ ] {n}{j}_{q}\rho
^{j}Q_{j,j}\left( x,y|\rho ,q\right) \sum_{m=0}^{n-j}\left( -1\right)
^{m+j}q^{\binom{m}{2}}\QATOPD[ ] {n-j}{m}_{q}\allowbreak = \\
q^{\binom{n}{2}}\rho ^{n}\left( 1-q\right) ^{n}Q_{n,n}\left( x,y|\rho
,q\right) 
\end{gather*}%
since $\forall n\geq 1:\sum_{i=0}^{n}(-1)^{i}q^{\binom{i}{2}}\QATOPD[ ] {n}{i%
}_{q}\allowbreak =\allowbreak 0.$
\end{proof}

\end{document}